\documentclass[12pt]{article}
\usepackage{currfile}
\usepackage{xcolor}
\usepackage{enumerate}
\usepackage{subfig}
\usepackage[justification=centering,labelsep=period]{caption}

\newtheorem{theorem}{Theorem}[section]

\newtheorem{conjecture}[theorem]{Conjecture}

\newtheorem{lemma}[theorem]{Lemma}

\newtheorem{claim}{Claim}

\newenvironment{proof}{\noindent {\bf
Proof.}}{\rule{3mm}{3mm}\par\medskip}
\newenvironment{pfc}{\noindent {\em
Proof.}}{\rule{2mm}{2mm}\par\medskip}

\usepackage{pgfplots}
\pgfplotsset{compat=1.15}
\usepackage{mathrsfs}
\usetikzlibrary{arrows}

\begin{document}

\title{Square Coloring of Planar Graphs with Maximum Degree at Most Five}

\author{Jiani Zou$^1$, Miaomiao Han$^1$, Hong-Jian Lai$^2$\\
\small $^1$College of Mathematical Science,  Tianjin Normal University, Tianjin 300387,  China,\\
\small $^2$Department of Mathematics, West Virginia University, Morgantown, WV, 26505\\
\small Emails: jiani\_zou@163.com; mmhan2018@hotmail.com; hjlai2015@hotmail.com\\
\small Dedicated to Professor Xueliang Li on the occasion of his 65th birthday.}

\date{}

\maketitle

\begin{abstract}
  The \textit{square} of a graph $G$, denoted by $G^2$, is obtained from $G$ by adding an edge to connect every pair of vertices with a common neighbor in $G$. In this paper we prove that for every planar graph $G$ with maximum degree at most $5$, $G^2$ admits a proper vertex coloring using at most $17$ colors, which improves the upper bound $18$ recently obtained by Hou, Jin, Miao, and Zhao in \cite{Hou2023}.
 
 \end{abstract}

\section{Introduction}

 In this paper, graphs are finite and simple. The edge set, vertex set,  minimum degree and maximum degree of a graph $G$ are denoted by $E(G), V(G)$, $\delta(G), \Delta(G)$, respectively. Let $G$ be a connected graph. For two distinct vertices $u,v\in V(G)$, denote by $d_G(u,v)$ the distance between vertices $u$ and $v$ in $G$. The \textit{square coloring} of $G$ is a mapping \(\kappa : V(G)\rightarrow \{1, 2, \ldots , \ell\}\) such that $\kappa(u)\not= \kappa(v)$ for any two distinct vertices $u,v$ with $d_{G}(u,v)\le 2$. For a vertex $v\in V(G)$, let $\kappa(v)$ denote the color of vertex $v$. For $V'\subseteq V(G)$, denote $\kappa(V')=\{\kappa(v): \forall v\in V'\}.$ The \textit{square} of $G$, denoted by $G^2$, is obtained from $G$ by adding an edge to connect every pair of vertices with a common neighbor in $G$. Obviously the square coloring of $G$ is a proper vertex coloring of $G^2$. The \textit{square chromatic number} of $G$, denoted by $\chi(G^2)$, is the minimum integer $\ell$ such that  $G$ admits a square coloring using $\ell$ colors.  For a graph $G$, we have $\Delta(G)+1\le\chi(G^2)\le\Delta(G^2)+1\le\Delta^2 (G)+1$. 
 
 The square coloring problem of graphs, proposed by Kramer and Kramer \cite{Kramer1969}, has been widely studied. In 1977, Wegner \cite{Wegner1977} proved that  $\chi(G^2)\le8$ if $G$ is a planar graph with $\Delta\le3$ and proposed a famous conjecture as stated below.
 
 \begin{conjecture}(Wegner)\label{c1}
 Every planar graph $G$ with maximum degree $\Delta$ satisfies
 $$ \chi (G^2)\le\left\{\begin{array}{ll}
			7, & if \Delta=3; \\
			\Delta+5, & if 4\le\Delta\le7;\\
			\lfloor \frac{3\Delta}{2}\rfloor+1,& if \Delta\ge8.
			
		\end{array}\right. $$
 \end{conjecture}

Thomassen \cite{Thomassen2018} proved that $\chi(G^2)\le7$ if $G$ is a planar graph with $\Delta=3$, which is also proved independently by Hartke, Jahanbekam and Thomas \cite{Hartke2016}. However, this conjecture remains open for planar graphs with $\Delta\ge4$. 

Wong \cite{Wong1996} showed that 
 $\chi(G^2)\le3\Delta+5$ for planar graphs with $\Delta\ge7$. Borodin Glebov, Van den Heuvel \cite{Borodin2002} proved that  $\chi(G^2)\le\lceil\frac{9\Delta}{5}\rceil+1$ for planar graphs with $\Delta\ge47$, Molloy and Salavatipour \cite{Molloy2005} improved this bound to $\chi(G^2)\le\lceil\frac{5\Delta}{3}\rceil+25$ for planar graphs with $\Delta\ge249$. Bousquet, de Meyer, Deschamps, and Pierron \cite{Bousquet2022} showed that $\chi(G^2)\le12$
 for planar graphs with $\Delta\le4$ by an automatic discharging method. Lih, Wang, and Zhu \cite{Lih2003} showed Conjecture \ref{c1} holds for $K_4$-minor free graphs. More results and related problems towards this conjecture can be found in the survey \cite{Cranston2023} given by Cranston. 
 
 Bousquet, Deschamps, de Meyer, and Pierron \cite{Bousquet2023} obtained the following result.

 \begin{theorem}\cite{Bousquet2023}\label{2delta+7}
 For any planar graph $G$ with $6\le\Delta\le31$, we have $\chi(G^2)\le2\Delta+7$.
 \end{theorem}
Note that Theorem \ref{2delta+7} does not include the case of $\Delta=5$.   For a planar graph $G$ with maximum degree at most $5$, Bu and Zhu \cite{Bu2018} proved $\chi(G^2)\le20$,  the upper bound was improved to 19 by Chen, Miao and Zhou \cite{Chen2022}, and recently Hou, Jin, Miao, and Zhao obtained the upper bound $18$ as below.
\begin{theorem}\cite{Hou2023}
 For any  planar graph $G$ with $\Delta\le5$, we have $\chi(G^2)\le 18$.
\end{theorem}

The main result of this paper further improves the upper bound to $17$ as stated in Theorem \ref{main} below. Furthermore, together with Theorem \ref{2delta+7} and some of the results listed above, Theorem \ref{main} implies that  $\chi(G^2)\le2\Delta+7$  holds for any planar graph $G$ with $\Delta\le 31$, which extends Theorem \ref{2delta+7}.

\begin{theorem}\label{main}
 For any  planar graph $G$ with $\Delta\le5$, we have $\chi(G^2)\le 17$.
\end{theorem}

The rest of this paper is organized as below. We introduce some terminology in Section 2, analyze the structure of a minimum counterexample of Theorem \ref{main} in Section 3.1 and apply discharging methods in Section 3.2 to obtain a contradiction, which justifies Theorem \ref{main}.

\section{Preliminaries}

To obtain Theorem \ref{main}, we first introduce  some notations and terms. For a vertex $u\in V(G)$, let $N_G(u)$ denote the set of neighbors of $u$, and let $d_{G}(u)=|N_G(u)|$ be the degree of $u$ in graph $G$. If $d_{G}(u)=k$ ($d_G(u)\ge k$, respectively), then we call $u$ a $k$-$vertex$ ($k^{+}$-$vertex$, respectively). For $u\in V(G)$, let $n_i(u)=|\{v\in N_G(u):d_G(v)=i\}|$ and $N^2_G(u)=\{v\in V(G)\setminus\{u\}:d_{G}(u,v)\le2\}$. Note that $N_G^2(u)$ represents the neighbors of $u$ in graph $G^2.$ If the graph $G$ is understood from context, we may omit the subscripts in those notations for convenience.

For a plane graph $G$, let $F(G)$ denote the set of all faces of $G$. For $f\in F(G)$, the degree of $f$, denoted by $d_G(f)$, is defined as the number of edges incident with the $f$, and $f$ is called a $d_{G}(f)$-face. If $f\in F(G)$ is a cycle with vertices $v_1,v_2,\cdots,v_{d_{G}(f)}$, then $f$ is denoted by $[v_1v_2\cdots v_{d_{G}(f)}].$ Let $f_i(v)$ denote the number of $i$-faces containing vertex $v$. If $d_G(f)=k$ ($d_G(f)\ge k$, resp.), then $f$ is called a $k$-face ($k^{+}$-face resp.).

For a $5$-vertex $v\in V(G)$, if $v$ is  contained in exactly four $3$-faces, then $v$ is called a {\em bad} $5$-vertex. For a bad $5$-vertex $v$, if $u\in N_G(v)$ is a $5$-vertex such that $uv$ is contained in both a $3$-face and a $5^{+}$-face, then we call $u$ a corner of $v$. Clearly, the number of corners of a bad $5$-vertex is at most two. For a $5$-vertex $u$ with $f_3(u)\le 3$, let $t_5(u)=|\{v\in N_G(u):$ $v$ is a bad  $5$-vertex and $u$ is a corner of $v\}|$.  

For any vertex $v\in V(G)$, it is straightforward to obtain that 
\begin{equation}\label{n2}
|N^2(v)|\le\sum_{u\in N_G(v)}d_G(u)-2f_3(v)-f_4(v).
\end{equation}

\section{Proof of Theorem \ref{main}}

In this section, we will present two key lemmas, Lemmas \ref{l0} and \ref{l1}, and analyze the structure of a minimum counterexample by developing several tool lemmas in Section 3.1 to prove Theorem \ref{main}. 

\vspace{0.1in}
\noindent{\textbf{ Theorem
\ref{main}}}
{\em For a planar graph $G$ with $\Delta\le5$, we have $\chi(G^2)\le 17$.}

\vspace{0.1in}

By contradiction, assume that Theorem \ref{main} has a counterexample. Choose one counterexample $G$ such that, among all counterexamples to Theorem \ref{main},
\begin{equation}\label{ex}
\mbox{$|E(G)|+|V(G)|$ is minimized.}
\end{equation}

Let $R=\{1,2,\dots,17\}$ be the set of colors. For a fixed vertex $v\in V(G)$ with $N_G(v)=\{v_1,v_2,\cdots,v_{d(v)}\}$. Let $\mathcal{G}(v) $ be a family of planar graphs such that a graph $M\in \mathcal{G}(v)$ if $M$ can be obtained from $G$ by deleting the vertex $v$ and adding some necessary edges $v_iv_j$ for some $i,j\in\{1,\cdots,d(v)\}$ satisfying each of the following:

(i) $d_{M}(v_i,v_j)\le 2$ for each pair $i,j\in \{1,\cdots,d(v)\}$;

(ii) $\Delta(M)\le 5$;

(iii) $|V(M)|+|E(M)|<|V(G)|+|E(G)|$.

\begin{lemma}\label{l0}
If there exists a planar graph $M\in \mathcal{G}(v)$ for a vertex $v$ in $G$, then $|N^2(v)|\ge 17.$
\end{lemma}
\begin{proof}
By contradiction, we assume $|N^2(v)|\le 16$. Take a planar graph $M\in \mathcal{G}(v)$ for a vertex $v\in V(G)$. Since $\Delta(M)\le 5$ and $|V(M)|+|E(M)|<|V(G)|+|E(G)|$, $M$ has a square coloring $\kappa$ with at most $17$ colors by the minimality of graph $G$. Since $|N^2(v)|\le 16$, we have $R\setminus \kappa(N^2(v))\not=\emptyset$. We define a new coloring $\tau$ of $G$ from $\kappa$ as below
$$ \tau (u)=\left\{\begin{array}{ll}
			\kappa(u), & u\in V(G)\setminus\{v\}; \\
			a\in R\setminus \kappa(N^2(v)), & u=v .
		\end{array}\right. $$
Now we are to prove that $\tau$ is a square coloring of $G$. Since $\tau(v)\in R\setminus \kappa(N^2(v))$, it suffices to prove that $\tau(x)\not=\tau(y)$ for arbitray two vertices $x,y\in V(G)\setminus \{v\}$ with $d_{G}(x,y)\le 2$. If $x,y\in N_G(v)$, then $d_M(x,y)\le 2$ by $M\in \mathcal{G}_v$, which implies that $\tau(x)=\kappa(x)\not=\kappa(y)=\tau(y).$ Otherwise, since $d_{M}(x,y)\le d_{G}(x,y)\le 2$, we have $\tau(x)=\kappa(x)\not=\kappa(y)=\tau(y).$ Therefore, $G$ has a square coloring $\tau$ with $\chi(G^2)\le17$, contrary to the fact that $G$ is a counterexample to Theorem \ref{main}. \end{proof}

\begin{lemma}\label{l1}
    For any edge $uv$ in $G$, we have either $\max\{|N^2(u)|,|N^2(v)|\}\ge 18$ or $\min\{|N^2(u)|,|N^2(v)|\}\ge 17$.
\end{lemma}
\begin{proof}
By contradiction, suppose that there exists an edge $uv\in E(G)$ such that $|N^2(u)|\le 17$ and $|N^2(v)|\le 16$ by assuming $|N^2(u)|\ge |N^2(v)|.$
Let $M=G-{uv}$. By the minimality of $G$ as stated in (\ref{ex}), $M$ has a square coloring $\kappa$ using at most 17 colors. Now we are to obtain a new coloring of $G$ from $\kappa$ by deleting colors of $u,v$ and recoloring $u,v$ successively. Since $|N^2(u)|\le 17$, we select one color $a\in R\setminus \kappa(N^2(u)\setminus v)$ to recolor $u$. By $|N^2(v)|\le 16$, we choose a color $b\in R\setminus \{\kappa(N^2(v)\setminus u)\cup a\}$ to recolor $v$. Thus we obtain a desired coloring of $G$ with $\chi(G^2)\le17$, contrary to the fact that $G$ is a counterexample to Theorem \ref{main}.\end{proof}

\subsection{Structural Properties of the Minimal Counterexample $G$}

In this subsection, we further explore the properties for this minimum counterexample $G$ of Theorem \ref{main}.

\begin{lemma}
 $\delta(G)\ge3.$
\end{lemma}
\begin{proof}
It is straightforward to obtain $\delta(G)\ge 2$. Assume by contradiction that there is $v\in V(G)$ with $d_G(v)=2$. Let $N_G(v)=\{v_1,v_2\}$. Make $M=G-v+\{v_1v_2\}$. By simple counting, we confirm $M\in \mathcal{G}(v)$. By (\ref{n2}), we have $|N^2(v)|\le \sum_{i=1}^{2}d_G(v_i)\le 10$, a contradiction to Lemma \ref{l0}. Thus $\delta(G)\ge3.$ \end{proof}
\begin{lemma}\label{l2}
    Let $[v_1v_2v_3]$ be a $3$-face in graph $G$. Then $d_G(v_i)\ge4$ for all $i$ with $1\le i\le 3 $. Moreover, there is some $i\in\{1,2,3\}$ with $d_G(v_i)=5$.
\end{lemma}
\begin{proof}
    By contradiction, assume $d_G(v_1)=3$ and let $v_4$ be a neighbor of $v_1$ other than $v_2,v_3$ (see Figure \ref{fig:1}(a)). Take $M=G-v_1+\{v_2v_4\}$. By simple computation, we have $d_M(v_i,v_j)\le 2$ for each pair $i,j\in \{2,3,4\}$,  $\Delta(M)\le\Delta(G)\le5$ and $|V(M)|+E(M)|\le |V(G)|+|E(G)|$. Thus we have $M\in \mathcal{G}(v_1)$. By (\ref{n2}), we get $|N^2(v_1)|\le \sum_{i=2}^{4}d_G(v_i)-2f_3(v_1)\le 13$, a contradiction to Lemma \ref{l0}.

     Next suppose $d_G(v_i)=4$ for each $i$ with $1 \le i\le 3$ and let the neighbors of $v_1$ be $v_2,v_3,v_5,v_4$ in the anticlockwise (see Figure \ref{fig:1}(b)). Let $M=G-v_1+\{v_2v_4,v_2v_5\}$. We get $M\in \mathcal{G}(v_1)$ by simple computation. By (\ref{n2}), we have  $|N^2(v_1)|\le\sum_{i=2}^{5}d_G(v_i)-2f_3(v_1)\le 16$, a contradiction to Lemma \ref{l0}.
    \end{proof}
\begin{figure}[h]
\centering
    \subfloat[]{
    \definecolor{aqaqaq}{rgb}{0.6274509803921569,0.6274509803921569,0.6274509803921569}
\begin{tikzpicture}[line cap=round,line join=round,>=triangle 45,x=0.5cm,y=0.5cm]
\clip(-3.,-2.5) rectangle (3.,3.3);
\draw [line width=0.8pt] (0.,0.)-- (-1.7320508075688774,-1.);
\draw [line width=0.8pt] (0.,0.)-- (1.7320508075688767,-1.);
\draw [line width=0.8pt] (0.,2.)-- (0.,0.);
\draw (0.06,0.6) node[anchor=north west] {$v_1$};
\draw (-2.7,-0.8) node[anchor=north west] {$v_2$};
\draw (1.6,-0.8) node[anchor=north west] {$v_3$};
\draw (-0.5,2.9) node[anchor=north west] {$v_4$};
\draw [line width=0.8pt,dotted,color=aqaqaq] (0.,2.)-- (-1.7320508075688774,-1.);
\draw [line width=0.8pt,dotted,color=aqaqaq] (0.,0.) circle (0.15cm);
\draw [line width=0.8pt] (-1.7320508075688774,-1.)-- (1.7320508075688767,-1.);
\begin{scriptsize}
\draw [fill=black] (0.,0.) circle (1.0pt);
\draw [fill=black] (0.,2.) circle (1.0pt);
\draw [fill=black] (-1.7320508075688774,-1.) circle (1.0pt);
\draw [fill=black] (1.7320508075688767,-1.) circle (1.0pt);
\end{scriptsize}
\end{tikzpicture}}
    \subfloat[]{
\definecolor{cqcqcq}{rgb}{0.7529411764705882,0.7529411764705882,0.7529411764705882}
\definecolor{aqaqaq}{rgb}{0.6274509803921569,0.6274509803921569,0.6274509803921569}
\begin{tikzpicture}[line cap=round,line join=round,>=triangle 45,x=0.5cm,y=0.5cm]
\clip(-3.,-3.) rectangle (3.,3.);
\draw [line width=0.8pt] (0.,0.)-- (-1.7320508075688774,-1.);
\draw [line width=0.8pt] (0.,0.)-- (1.7320508075688767,-1.);
\draw (-0.5,1.1) node[anchor=north west] {$v_1$};
\draw (-2.8,-0.8) node[anchor=north west] {$v_2$};
\draw (1.6,-0.8) node[anchor=north west] {$v_3$};
\draw (-2.9,1.8) node[anchor=north west] {$v_4$};
\draw [line width=0.8pt] (-1.7320508075688772,1.)-- (0.,0.);
\draw [line width=0.8pt] (0.,0.)-- (1.7320508075688776,1.);
\draw (1.6,1.7) node[anchor=north west] {$v_5$};
\draw [line width=0.8pt,dotted,color=aqaqaq] (-1.7320508075688772,1.)-- (-1.7320508075688774,-1.);
\draw [line width=0.8pt,dotted,color=cqcqcq] (0.,0.) circle (0.15cm);
\draw [shift={(-1.4460409011024125,2.449750758907173)},line width=0.8pt,dotted,color=aqaqaq]  plot[domain=4.621629793443702:5.851372033610908,variable=\t]({1.*3.4640078831263077*cos(\t r)+0.*3.4640078831263077*sin(\t r)},{0.*3.4640078831263077*cos(\t r)+1.*3.4640078831263077*sin(\t r)});
\draw [line width=0.8pt] (-1.7320508075688774,-1.)-- (1.7320508075688767,-1.);
\begin{scriptsize}
\draw [fill=black] (0.,0.) circle (1.0pt);
\draw [fill=black] (-1.7320508075688774,-1.) circle (1.0pt);
\draw [fill=black] (1.7320508075688767,-1.) circle (1.0pt);
\draw [fill=black] (-1.7320508075688772,1.) circle (1.0pt);
\draw [fill=black] (1.7320508075688776,1.) circle (1.0pt);
\end{scriptsize}
\end{tikzpicture}}
    \subfloat[]{
\definecolor{aqaqaq}{rgb}{0.6274509803921569,0.6274509803921569,0.6274509803921569}
\begin{tikzpicture}[line cap=round,line join=round,>=triangle 45,x=0.5cm,y=0.5cm]
\clip(-3.,-2.5) rectangle (3.,3.3);
\draw [line width=0.8pt] (0.,0.)-- (-1.7320508075688774,-1.);
\draw [line width=0.8pt] (0.,0.)-- (1.7320508075688767,-1.);
\draw [line width=0.8pt] (0.,2.)-- (0.,0.);
\draw (0.06,0.6) node[anchor=north west] {$v$};
\draw (-2.7,-0.8) node[anchor=north west] {$v_2$};
\draw (1.6,-0.8) node[anchor=north west] {$v_3$};
\draw (-0.5,2.9) node[anchor=north west] {$v_1$};
\draw [line width=0.8pt,dotted,color=aqaqaq] (0.,2.)-- (-1.7320508075688774,-1.);
\draw [line width=0.8pt,dotted,color=aqaqaq] (0.,0.) circle (0.15cm);
\draw [line width=0.8pt,dotted,color=aqaqaq] (0.,2.)-- (1.7320508075688767,-1.);
\begin{scriptsize}
\draw [fill=black] (0.,0.) circle (1.0pt);
\draw [fill=black] (0.,2.) circle (1.0pt);
\draw [fill=black] (-1.7320508075688774,-1.) circle (1.0pt);
\draw [fill=black] (1.7320508075688767,-1.) circle (1.0pt);
\end{scriptsize}
\end{tikzpicture}}
    \subfloat[]{
\definecolor{aqaqaq}{rgb}{0.6274509803921569,0.6274509803921569,0.6274509803921569}
\begin{tikzpicture}[line cap=round,line join=round,>=triangle 45,x=0.5cm,y=0.5cm]
\clip(-3.,-3.) rectangle (3.,3.);
\draw [line width=0.8pt] (0.,0.)-- (-1.7320508075688774,-1.);
\draw [line width=0.8pt] (0.,0.)-- (1.7320508075688767,-1.);
\draw [line width=0.8pt] (0.,2.)-- (0.,0.);
\draw (0.06355179567919868,0.5250960380835304) node[anchor=north west] {$v$};
\draw (-2.7,-0.8) node[anchor=north west] {$v_2$};
\draw (1.6,-0.8) node[anchor=north west] {$v_3$};
\draw (-0.5,2.9) node[anchor=north west] {$v_1$};
\draw [line width=0.8pt,dotted,color=aqaqaq] (0.,0.) circle (0.15cm);
\draw [line width=0.8pt,dotted,color=aqaqaq] (0.,2.)-- (1.7320508075688767,-1.);
\draw [line width=0.8pt] (-1.7320508075688774,1.)-- (-1.7320508075688774,-1.);
\draw [line width=0.8pt] (-1.7320508075688774,1.)-- (0.,2.);
\draw [line width=0.8pt] (-1.7320508075688774,-1.)-- (0.,-2.);
\draw [line width=0.8pt] (0.,-2.)-- (1.7320508075688767,-1.);
\draw (-2.6,1.5) node[anchor=north west] {$x$};
\draw (-0.4,-1.9278102105154634) node[anchor=north west] {$y$};
\begin{scriptsize}
\draw [fill=black] (0.,0.) circle (1.0pt);
\draw [fill=black] (0.,2.) circle (1.0pt);
\draw [fill=black] (-1.7320508075688774,-1.) circle (1.0pt);
\draw [fill=black] (1.7320508075688767,-1.) circle (1.0pt);
\draw [fill=black] (-1.7320508075688774,1.) circle (1.0pt);
\draw [fill=black] (0.,-2.) circle (1.0pt);
\end{scriptsize}
\end{tikzpicture}
}
\caption{The graphs of Lemmas \ref{l2} and \ref{l3}}
    \label{fig:1}
\end{figure}
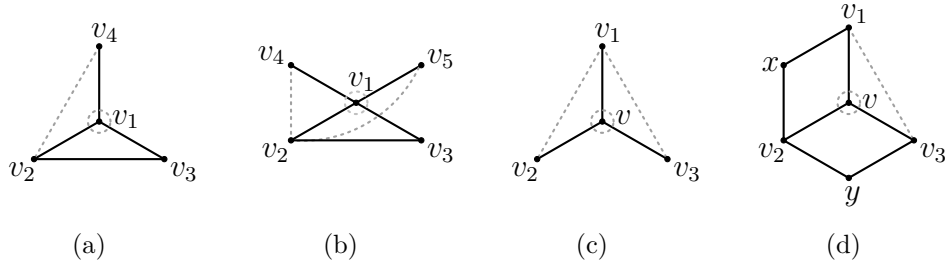
    
\begin{lemma}\label{l3}
Let $v\in V(G)$ be a $3$-vertex with $N_G(v)=\{v_1,v_2,v_3\}$. Then $d_G(v_i)=5$ for all $i$ with $ 1\le i\le3$ and $f_4(v)\le1$. 
\end{lemma}
\begin{proof}
Assume by contradiction that $d_G(v_1)\le4$. Let $M=G-v+\{v_1v_2,v_1v_3\}$(see Figure \ref{fig:1}(c)). We have $d_M(v_i)=d_G(v_i)$ for $2\le i\le3$, $d_M(v_1)=d_G(v_1)+1\le5$. Thus $M\in \mathcal{G}(v)$. By (\ref{n2}), $ |N^2(v)|\le\sum_{i=1}^{i=3}d_{G}(v_i)\le 14$, a contradiction to Lemma \ref{l0}.

Next suppose $f_4(v)=2$ and denote by $[vv_1xv_2]$ and $[vv_2yv_3]$ the two $4$-faces (see Figure \ref{fig:1}(d)). Let $M=G-v+\{v_1v_3\}\in \mathcal{G}(v)$. We obtain a contradiction to Lemma \ref{l0}, since $ |N^2(v)|\le \sum_{i=1}^{3}d_{G}(v_i)-f_{4}(v)\le 13$ by (\ref{n2}). \end{proof}

\begin{lemma}\label{l4}
Let $v\in V(G)$ be a $k$-vertex with $N_G(v)=\{v_1,v_2,\dots, v_k\}$. Then we have

(i) If $k=4$, then $f_3(v)\le1$; 

(ii) If $k=5$, then $f_3(v)\le4$.
\end{lemma}
\begin{proof}
(i) By contradiction, assume that $f_3(v)=2$. There are two cases to discuss according to whether the two $3$-faces are edge-disjoint or not. If there are two $3$-faces $[vv_1v_2],[vv_2v_3]$ in $G$ (see Figure \ref{fig:2}(a)), then let $M=G-v+\{v_2v_4\}$. Thus $M\in \mathcal{G}(v)$. By (\ref{n2}), we have $ |N^2(v)|\le\sum_{i=1}^{i=4}d_{G}(v_i)-2f_3(v)\le 16$, a contradiction to Lemma \ref{l0}. 

\begin{figure}[h]
\centering
    \subfloat[]{
\definecolor{aqaqaq}{rgb}{0.6274509803921569,0.6274509803921569,0.6274509803921569}
\begin{tikzpicture}[line cap=round,line join=round,>=triangle 45,x=0.5cm,y=0.5cm]
\clip(-3.,-3.) rectangle (3.,3.);
\draw [line width=0.8pt] (0.,2.)-- (0.,-2.);
\draw [line width=0.8pt] (-2.,0.)-- (2.,0.);
\draw (2,0.5) node[anchor=north west] {$v_1$};
\draw (-0.1,0.8) node[anchor=north west] {$v$};
\draw (-0.5,2.9) node[anchor=north west] {$v_2$};
\draw (-3.3,0.5) node[anchor=north west] {$v_3$};
\draw (-0.5,-1.9) node[anchor=north west] {$v_4$};
\draw [line width=0.8pt] (0.,2.)-- (-2.,0.);
\draw [line width=0.8pt] (0.,2.)-- (2.,0.);
\draw [line width=0.8pt,dotted,color=aqaqaq] (0.,0.) circle (0.15cm);
\draw [shift={(-2.4133333333333336,0.)},line width=0.8pt,dotted,color=aqaqaq]  plot[domain=-0.6920149713822452:0.6920149713822449,variable=\t]({1.*3.134354443546195*cos(\t r)+0.*3.134354443546195*sin(\t r)},{0.*3.134354443546195*cos(\t r)+1.*3.134354443546195*sin(\t r)});
\begin{scriptsize}
\draw [fill=black] (0.,0.) circle (1.0pt);
\draw [fill=black] (-2.,0.) circle (1.0pt);
\draw [fill=black] (2.,0.) circle (1.0pt);
\draw [fill=black] (0.,2.) circle (1.0pt);
\draw [fill=black] (0.,-2.) circle (1.0pt);
\end{scriptsize}
\end{tikzpicture}}
\hspace{2em}
    \subfloat[]{
    \definecolor{aqaqaq}{rgb}{0.6274509803921569,0.6274509803921569,0.6274509803921569}
\begin{tikzpicture}[line cap=round,line join=round,>=triangle 45,x=0.5cm,y=0.5cm]
\clip(-3.,-3.) rectangle (3.,3.);
\draw [line width=0.8pt] (0.,2.)-- (0.,-2.);
\draw [line width=0.8pt] (-2.,0.)-- (2.,0.);
\draw (2,0.5) node[anchor=north west] {$v_1$};
\draw (0.06,0.8) node[anchor=north west] {$v$};
\draw (-0.5,2.9) node[anchor=north west] {$v_2$};
\draw (-3.3,0.5) node[anchor=north west] {$v_3$};
\draw (-0.5,-1.9) node[anchor=north west] {$v_4$};
\draw [line width=0.8pt] (0.,2.)-- (2.,0.);
\draw [line width=0.8pt,dotted,color=aqaqaq] (0.,0.) circle (0.15cm);
\draw [line width=0.8pt] (-2.,0.)-- (0.,-2.);
\draw [line width=0.8pt,dotted,color=aqaqaq] (0.,2.)-- (-2.,0.);
\draw [line width=0.8pt,dotted,color=aqaqaq] (2.,0.)-- (0.,-2.);
\begin{scriptsize}
\draw [fill=black] (0.,0.) circle (1.0pt);
\draw [fill=black] (-2.,0.) circle (1.0pt);
\draw [fill=black] (2.,0.) circle (1.0pt);
\draw [fill=black] (0.,2.) circle (1.0pt);
\draw [fill=black] (0.,-2.) circle (1.0pt);
\end{scriptsize}
\end{tikzpicture}}
\hspace{2em}
    \subfloat[]{
    \definecolor{aqaqaq}{rgb}{0.6274509803921569,0.6274509803921569,0.6274509803921569}
\definecolor{sqsqsq}{rgb}{0.12549019607843137,0.12549019607843137,0.12549019607843137}
\begin{tikzpicture}[line cap=round,line join=round,>=triangle 45,x=0.6cm,y=0.6cm]
\clip(-2.5,-2.) rectangle (2.5,3.);
\draw [line width=0.8pt] (0.,2.077683537175253)-- (0.,0.37638192047117347);
\draw [line width=0.8pt] (-1.6180339887498947,0.9021130325903073)-- (0.,0.37638192047117347);
\draw [line width=0.8pt] (0.,0.37638192047117347)-- (-1.,-1.);
\draw [line width=0.8pt] (0.,0.37638192047117347)-- (1.,-1.);
\draw [line width=0.8pt] (0.,0.37638192047117347)-- (1.618033988749895,0.9021130325903065);
\draw (1.5,1.3) node[anchor=north west] {$v_1$};
\draw (-0.5,2.9) node[anchor=north west] {$v_2$};
\draw (-2.6,1.2679725201426373) node[anchor=north west] {$v_3$};
\draw (-1.6,-0.9000098641126053) node[anchor=north west] {$v_4$};
\draw (0.8,-0.8545277161911666) node[anchor=north west] {$v_5$};
\draw [line width=0.8pt] (-1.,-1.)-- (1.,-1.);
\draw [line width=0.8pt] (-1.6180339887498947,0.9021130325903073)-- (-1.,-1.);
\draw [line width=0.8pt] (-1.6180339887498947,0.9021130325903073)-- (0.,2.077683537175253);
\draw [line width=0.8pt] (0.,2.077683537175253)-- (1.618033988749895,0.9021130325903065);
\draw [line width=0.8pt] (1.618033988749895,0.9021130325903065)-- (1.,-1.);
\draw (0.08646136107768357,0.6767045971639347) node[anchor=north west] {$v$};
\draw [line width=0.8pt,dotted,color=aqaqaq] (0.,0.37638192047117347) circle (0.18cm);
\begin{scriptsize}
\draw [fill=sqsqsq] (-1.,-1.) circle (1.0pt);
\draw [fill=sqsqsq] (1.,-1.) circle (1.0pt);
\draw [fill=sqsqsq] (1.618033988749895,0.9021130325903065) circle (1.0pt);
\draw [fill=sqsqsq] (0.,2.077683537175253) circle (1.0pt);
\draw [fill=sqsqsq] (-1.6180339887498947,0.9021130325903073) circle (1.0pt);
\draw [fill=sqsqsq] (0.,0.37638192047117347) circle (1.0pt);
\end{scriptsize}
\end{tikzpicture}}
    
    \caption{The graphs of Lemma \ref{l4}}
    \label{fig:2}
\end{figure}
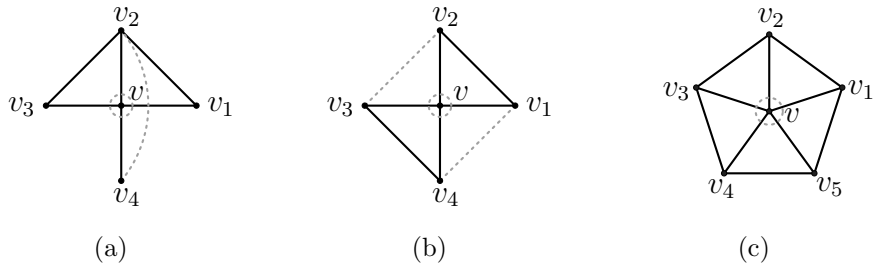

Otherwise, let $[vv_1v_2]$ and $[vv_3v_4]$ be the two $3$-$faces$  (see Figure \ref{fig:2}(b)). We can verify that $M=G-v+\{v_1v_4,v_2v_3\}\in \mathcal{G}(v)$. By (\ref{n2}), we have $ |N^2(v)|\le\sum_{i=1}^{i=4}d_{G}(v_i)-2f_3(v)\le 16$, which is a contradiction to Lemma \ref{l0}. 

(ii) Then we prove $f_3(v)\le4$, if $d_G(v)=5$. By contradiction, suppose $f_3(v)=5$ and we have $v_5v_1,v_iv_{i+1}\in E(G)$ for each $i$ with $1\le i\le 4$ (see Figure \ref{fig:2}(c)). Let $M=G-v$. Then, it is routine to verify $M\in \mathcal{G}(v)$, and so we have $ |N^2(v)|\le\sum_{i=1}^{i=5}d_{G}(v_i)-2f_3(v)\le 15$ by (\ref{n2}). There is a contradiction to Lemma \ref{l0}.\end{proof}

\begin{lemma}\label{l5}
Let $v\in V(G)$ be a $4$-vertex with $N_G(v)=\{v_1,v_2,v_3,v_4\}$. If $v$ is contained in one triangle $[vv_1v_2]$, then we have

(i) $f_4(v)\le2.$ 

(ii) If $f_4(v)=2$, then $d(v_i)=5$ for all i with $ 1\le i\le 4$.
\end{lemma}
\begin{proof}
(i) First we prove $f_4(v)\le2$, when $f_3(v)=1$. If $f_4(v)=3$, then denote by $[vv_2xv_3],[vv_3yv_4]$ $[vv_4zv_1]$ the three $4$-faces (see Figure \ref{fig:3}(a)). Let $M=G-v+\{v_2v_3,v_1v_4\}$. We can verify $M\in \mathcal{G}(v)$. By (\ref{n2}), we have $ |N^2(v)|\le\sum_{i=1}^{i=4}d_{G}(v_i)-f_4(v)-2f_3(v)\le 15$,
a contradiction to Lemma \ref{l0}.



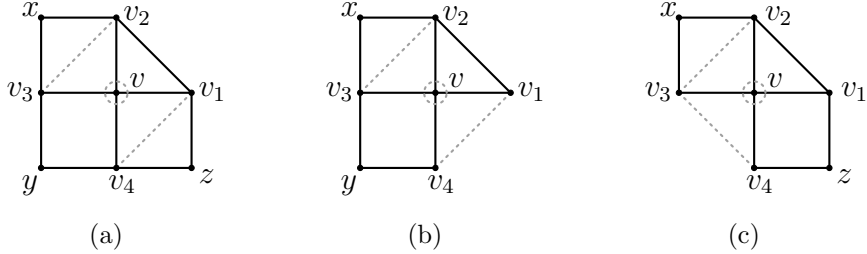
\begin{figure}[h]
\centering
    \subfloat[]{
    \definecolor{aqaqaq}{rgb}{0.6274509803921569,0.6274509803921569,0.6274509803921569}
\begin{tikzpicture}[line cap=round,line join=round,>=triangle 45,x=0.5cm,y=0.5cm]
\clip(-3.,-3.) rectangle (3.,3.);
\draw [line width=0.8pt] (0.,2.)-- (0.,-2.);
\draw [line width=0.8pt] (-2.,0.)-- (2.,0.);
\draw (1.9,0.56) node[anchor=north west] {$v_1$};
\draw (0.06,0.8) node[anchor=north west] {$v$};
\draw (-0.1,2.58) node[anchor=north west] {$v_2$};
\draw (-3.2,0.56) node[anchor=north west] {$v_3$};
\draw (-0.5,-1.9) node[anchor=north west] {$v_4$};
\draw [line width=0.8pt] (-2.,2.)-- (-2.,-2.);
\draw [line width=0.8pt] (-2.,-2.)-- (2.,-2.);
\draw [line width=0.8pt] (-2.,2.)-- (0.,2.);
\draw [line width=0.8pt] (2.,0.)-- (2.,-2.);
\draw [line width=0.8pt] (2.,0.)-- (0,2.);
\draw (-2.8,2.7) node[anchor=north west] {$x$};
\draw (-2.8,-1.9) node[anchor=north west] {$y$};
\draw (1.9,-1.76) node[anchor=north west] {$z$};
\draw [line width=0.8pt,dotted,color=aqaqaq] (0.,2.)-- (-2.,0.);
\draw [line width=0.8pt,dotted,color=aqaqaq] (2.,0.)-- (0.,-2.);
\draw [line width=0.8pt,dotted,color=aqaqaq] (0.,0.) circle (0.15cm);
\begin{scriptsize}
\draw [fill=black] (0.,0.) circle (1.0pt);
\draw [fill=black] (-2.,0.) circle (1.0pt);
\draw [fill=black] (2.,0.) circle (1.0pt);
\draw [fill=black] (0.,2.) circle (1.0pt);
\draw [fill=black] (0.,-2.) circle (1.0pt);
\draw [fill=black] (-2.,2.) circle (1.0pt);
\draw [fill=black] (-2.,-2.) circle (1.0pt);
\draw [fill=black] (2.,-2.) circle (1.0pt);
\end{scriptsize}
\end{tikzpicture}}
\hspace{2em}
    \subfloat[]{
    \definecolor{aqaqaq}{rgb}{0.6274509803921569,0.6274509803921569,0.6274509803921569}
\begin{tikzpicture}[line cap=round,line join=round,>=triangle 45,x=0.5cm,y=0.5cm]
\clip(-3.,-3.) rectangle (3.,3.);
\draw [line width=0.8pt] (0.,2.)-- (0.,-2.);
\draw [line width=0.8pt] (-2.,0.)-- (2.,0.);
\draw (1.9,0.56) node[anchor=north west] {$v_1$};
\draw (0.06,0.8) node[anchor=north west] {$v$};
\draw (-0.1,2.58) node[anchor=north west] {$v_2$};
\draw (-3.2,0.56) node[anchor=north west] {$v_3$};
\draw (-0.5,-1.9) node[anchor=north west] {$v_4$};
\draw [line width=0.8pt] (-2.,2.)-- (-2.,-2.);
\draw [line width=0.8pt] (-2.,2.)-- (0.,2.);
\draw [line width=0.8pt] (2.,0.)-- (0,2.);
\draw (-2.8,2.7) node[anchor=north west] {$x$};
\draw (-2.8,-1.9) node[anchor=north west] {$y$};
\draw [line width=0.8pt,dotted,color=aqaqaq] (0.,2.)-- (-2.,0.);
\draw [line width=0.8pt,dotted,color=aqaqaq] (2.,0.)-- (0.,-2.);
\draw [line width=0.8pt,dotted,color=aqaqaq] (0.,0.) circle (0.15cm);
\draw [line width=0.8pt] (-2.,-2.)-- (0.,-2.);
\begin{scriptsize}
\draw [fill=black] (0.,0.) circle (1.0pt);
\draw [fill=black] (-2.,0.) circle (1.0pt);
\draw [fill=black] (2.,0.) circle (1.0pt);
\draw [fill=black] (0.,2.) circle (1.0pt);
\draw [fill=black] (0.,-2.) circle (1.0pt);
\draw [fill=black] (-2.,2.) circle (1.0pt);
\draw [fill=black] (-2.,-2.) circle (1.0pt);
\end{scriptsize}
\end{tikzpicture}}
\hspace{2em}
    \subfloat[]{
    \definecolor{aqaqaq}{rgb}{0.6274509803921569,0.6274509803921569,0.6274509803921569}
\begin{tikzpicture}[line cap=round,line join=round,>=triangle 45,x=0.5cm,y=0.5cm]
\clip(-3.,-3.) rectangle (3.,3.);
\draw [line width=0.8pt] (0.,2.)-- (0.,-2.);
\draw [line width=0.8pt] (-2.,0.)-- (2.,0.);
\draw (2.02,0.56) node[anchor=north west] {$v_1$};
\draw (0.06,0.8) node[anchor=north west] {$v$};
\draw (0.,2.58) node[anchor=north west] {$v_2$};
\draw (-3.2,0.56) node[anchor=north west] {$v_3$};
\draw (-0.5,-1.9) node[anchor=north west] {$v_4$};
\draw [line width=0.8pt] (-2.,2.)-- (0.,2.);
\draw [line width=0.8pt] (2.,0.)-- (2.,-2.);
\draw [line width=0.8pt] (2.,0.)-- (0,2.);
\draw (-2.8,2.7) node[anchor=north west] {$x$};
\draw (1.9,-1.76) node[anchor=north west] {$z$};
\draw [line width=0.8pt,dotted,color=aqaqaq] (0.,2.)-- (-2.,0.);
\draw [line width=0.8pt,dotted,color=aqaqaq] (0.,0.) circle (0.15cm);
\draw [line width=0.8pt] (-2.,2.)-- (-2.,0.);
\draw [line width=0.8pt] (0.,-2.)-- (2.,-2.);
\draw [line width=0.8pt,dotted,color=aqaqaq] (-2.,0.)-- (0.,-2.);
\begin{scriptsize}
\draw [fill=black] (0.,0.) circle (1.0pt);
\draw [fill=black] (-2.,0.) circle (1.0pt);
\draw [fill=black] (2.,0.) circle (1.0pt);
\draw [fill=black] (0.,2.) circle (1.0pt);
\draw [fill=black] (0.,-2.) circle (1.0pt);
\draw [fill=black] (-2.,2.) circle (1.0pt);
\draw [fill=black] (2.,-2.) circle (1.0pt);
\end{scriptsize}
\end{tikzpicture}}
    \caption{The graphs of Lemma \ref{l5}}
    \label{fig:3}
\end{figure}

(ii) If $v$ is contained in one triangle and two $4$-$faces$, then there will be two cases as the two $4$-faces may share an edge or may be edge disjoint. Let $[vv_2xv_3],[vv_3yv_4]$ be the two $4$-faces in $G$ (see Figure \ref{fig:3}(b)). Assume by contradiction that there is some $i\in\{1, 2, 3, 4\}$ such that $d_G(v_i)\le4$. Then $M=G-v+\{v_2v_3,v_1v_4\}\in \mathcal{G}(v)$. We obtain a contradiction to Lemma \ref{l0}, since $ |N^2(v)|\le \sum_{i=1}^{i=4}d_{G}(v_i)-2f_{3}(v)-f_{4}(v)\le 15$.

Therefore, we may assume that $[vv_2xv_3],[vv_4zv_1]$ are the two $4$-faces (see Figure \ref{fig:3}(c)). By symmetry, it suffices to prove that $d_G(v_2)=d_G(v_3)=5$. If $d_{G}(v_2)\le 4$, then $|N^2(v)|=\sum_{i=1}^{i=4}d_{G}(v_i)-2f_{3}(v)-f_{4}(v) \le 15$ and $|N^2(v_2)|=\sum_{u\in N_G(v_2))}d_G(u)-2f_3(v_2)-f_4(v_2)\le 16$, a contradiction to Lemma \ref{l1}. Thus $d_G(v_2)=5$. If $d_G(v_3)\le4$, then let $M=G-v+\{v_2v_3,v_3v_4\}$. By verifying $M\in \mathcal{G}(v)$, we have $ |N^2(v)|\le \sum_{i=1}^{i=4}d_{G}(v_i)-2f_{3}(v)-f_{4}(v)\le 15$, contrary to Lemma \ref{l0}.\end{proof}

\begin{lemma}\label{l6}
Let $v\in V(G)$ be a $5$-vertex with $N_G(v)=\{v_1,v_2,v_3,v_4,v_5\}$. Then each of the following holds:

(i) If $f_3(v)=2$, then $n_3(v)\le1$.
       
(ii) If $f_3(v)=3$, then $n_3(v)=0$. Moreover, (a) if $f_3(v)=3, f_4(v)=2$, then $n_4(v)=0$; (b) if $f_3(v)=3, f_4(v)=1$, then $n_4(v)\le1$.

(iii) If $f_3(v)=4$, then $f_4(v)=0$ and $d_G(v_i)=5$ for all $i$ with $1\le i\le5$. Moreover, every corner of bad $5$-vertex $v$ is not a bad $5$-vertex.
    
    \end{lemma}
    \begin{figure}[h]
    \centering
    \subfloat[]{
    \definecolor{sqsqsq}{rgb}{0.12549019607843137,0.12549019607843137,0.12549019607843137}
\begin{tikzpicture}[line cap=round,line join=round,>=triangle 45,x=0.6cm,y=0.6cm]
\clip(-2.5,-2.) rectangle (2.5,3.);
\draw [line width=0.8pt] (0.,2.077683537175253)-- (0.,0.37638192047117347);
\draw [line width=0.8pt] (-1.6180339887498947,0.9021130325903073)-- (0.,0.37638192047117347);
\draw [line width=0.8pt] (0.,0.37638192047117347)-- (-1.,-1.);
\draw [line width=0.8pt] (0.,0.37638192047117347)-- (1.,-1.);
\draw [line width=0.8pt] (0.,0.37638192047117347)-- (1.618033988749895,0.9021130325903065);
\draw (1.5,1.3) node[anchor=north west] {$v_1$};
\draw (-0.5,2.9) node[anchor=north west] {$v_2$};
\draw (-2.6,1.2679725201426373) node[anchor=north west] {$v_3$};
\draw (-1.6,-0.9000098641126053) node[anchor=north west] {$v_4$};
\draw (0.8,-0.8545277161911666) node[anchor=north west] {$v_5$};
\draw [line width=0.8pt] (-1.,-1.)-- (1.,-1.);
\draw [line width=0.8pt] (-1.6180339887498947,0.9021130325903073)-- (-1.,-1.);
\draw (0.046739112826465715,0.5431669413394702) node[anchor=north west] {$v$};
\draw [line width=1.2pt] (0.6914269590910433,0.6876647115976507)-- (0.8726254405972135,0.5766075777712885);
\draw [line width=1.2pt] (0.8236584996586465,0.7273123257637493)-- (0.7423435941902274,0.5309594219845322);
\begin{scriptsize}
\draw [fill=sqsqsq] (-1.,-1.) circle (1.0pt);
\draw [fill=sqsqsq] (1.,-1.) circle (1.0pt);
\draw [fill=sqsqsq] (1.618033988749895,0.9021130325903065) circle (1.0pt);
\draw [fill=sqsqsq] (0.,2.077683537175253) circle (1.0pt);
\draw [fill=sqsqsq] (-1.6180339887498947,0.9021130325903073) circle (1.0pt);
\draw [fill=sqsqsq] (0.,0.37638192047117347) circle (1.0pt);
\end{scriptsize}
\end{tikzpicture}}
    \subfloat[]{
    \definecolor{sqsqsq}{rgb}{0.12549019607843137,0.12549019607843137,0.12549019607843137}
\begin{tikzpicture}[line cap=round,line join=round,>=triangle 45,x=0.6cm,y=0.6cm]
\clip(-2.5,-2.) rectangle (2.5,3.);
\draw [line width=0.8pt] (0.,2.077683537175253)-- (0.,0.37638192047117347);
\draw [line width=0.8pt] (-1.6180339887498947,0.9021130325903073)-- (0.,0.37638192047117347);
\draw [line width=0.8pt] (0.,0.37638192047117347)-- (-1.,-1.);
\draw [line width=0.8pt] (0.,0.37638192047117347)-- (1.,-1.);
\draw [line width=0.8pt] (0.,0.37638192047117347)-- (1.618033988749895,0.9021130325903065);
\draw (1.5,1.3) node[anchor=north west] {$v_1$};
\draw (-0.5,2.9) node[anchor=north west] {$v_2$};
\draw (-2.6,1.2679725201426373) node[anchor=north west] {$v_3$};
\draw (-1.6,-0.9000098641126053) node[anchor=north west] {$v_4$};
\draw (0.8,-0.8545277161911666) node[anchor=north west] {$v_5$};
\draw [line width=0.8pt] (-1.,-1.)-- (1.,-1.);
\draw (0.09689968582943193,0.6) node[anchor=north west] {$v$};
\draw [line width=1.2pt] (0.6914269590910433,0.6876647115976507)-- (0.8726254405972135,0.5766075777712885);
\draw [line width=1.2pt] (0.8236584996586465,0.7273123257637493)-- (0.7423435941902274,0.5309594219845322);
\draw [line width=0.8pt] (0.,2.077683537175253)-- (-1.6180339887498947,0.9021130325903073);
\begin{scriptsize}
\draw [fill=sqsqsq] (-1.,-1.) circle (1.0pt);
\draw [fill=sqsqsq] (1.,-1.) circle (1.0pt);
\draw [fill=sqsqsq] (1.618033988749895,0.9021130325903065) circle (1.0pt);
\draw [fill=sqsqsq] (0.,2.077683537175253) circle (1.0pt);
\draw [fill=sqsqsq] (-1.6180339887498947,0.9021130325903073) circle (1.0pt);
\draw [fill=sqsqsq] (0.,0.37638192047117347) circle (1.0pt);
\end{scriptsize}
\end{tikzpicture}}
    \subfloat[]{
    \definecolor{aqaqaq}{rgb}{0.6274509803921569,0.6274509803921569,0.6274509803921569}
\definecolor{sqsqsq}{rgb}{0.12549019607843137,0.12549019607843137,0.12549019607843137}
\begin{tikzpicture}[line cap=round,line join=round,>=triangle 45,x=0.6cm,y=0.6cm]
\clip(-2.5,-2.) rectangle (2.5,3.);
\draw [line width=0.8pt] (0.,2.077683537175253)-- (0.,0.37638192047117347);
\draw [line width=0.8pt] (-1.6180339887498947,0.9021130325903073)-- (0.,0.37638192047117347);
\draw [line width=0.8pt] (0.,0.37638192047117347)-- (-1.,-1.);
\draw [line width=0.8pt] (0.,0.37638192047117347)-- (1.,-1.);
\draw [line width=0.8pt] (0.,0.37638192047117347)-- (1.618033988749895,0.9021130325903065);
\draw (1.5,1.3) node[anchor=north west] {$v_1$};
\draw (-0.5,2.9) node[anchor=north west] {$v_2$};
\draw (-2.6,1.2679725201426373) node[anchor=north west] {$v_3$};
\draw (-1.6,-0.9000098641126053) node[anchor=north west] {$v_4$};
\draw (0.8,-0.8545277161911666) node[anchor=north west] {$v_5$};
\draw [line width=0.8pt] (-1.,-1.)-- (1.,-1.);
\draw [line width=0.8pt,dotted,color=aqaqaq] (-1.6180339887498947,0.9021130325903073)-- (-1.,-1.);
\draw [line width=0.8pt] (-1.6180339887498947,0.9021130325903073)-- (0.,2.077683537175253);
\draw [line width=0.8pt] (0.,2.077683537175253)-- (1.618033988749895,0.9021130325903065);
\draw [line width=0.8pt,dotted,color=aqaqaq] (1.618033988749895,0.9021130325903065)-- (1.,-1.);
\draw (0.07267889201058018,0.6331018768425571) node[anchor=north west] {$v$};
\draw [line width=0.8pt,dotted,color=aqaqaq] (0.,0.37638192047117347) circle (0.18cm);
\begin{scriptsize}
\draw [fill=sqsqsq] (-1.,-1.) circle (1.0pt);
\draw [fill=sqsqsq] (1.,-1.) circle (1.0pt);
\draw [fill=sqsqsq] (1.618033988749895,0.9021130325903065) circle (1.0pt);
\draw [fill=sqsqsq] (0.,2.077683537175253) circle (1.0pt);
\draw [fill=sqsqsq] (-1.6180339887498947,0.9021130325903073) circle (1.0pt);
\draw [fill=sqsqsq] (0.,0.37638192047117347) circle (1.0pt);
\end{scriptsize}
\end{tikzpicture}}
    \subfloat[]{
    \definecolor{aqaqaq}{rgb}{0.6274509803921569,0.6274509803921569,0.6274509803921569}
\definecolor{sqsqsq}{rgb}{0.12549019607843137,0.12549019607843137,0.12549019607843137}
\begin{tikzpicture}[line cap=round,line join=round,>=triangle 45,x=0.6cm,y=0.6cm]
\clip(-2.5,-2.) rectangle (2.5,3.);
\draw [line width=0.8pt] (0.,2.077683537175253)-- (0.,0.37638192047117347);
\draw [line width=0.8pt] (-1.6180339887498947,0.9021130325903073)-- (0.,0.37638192047117347);
\draw [line width=0.8pt] (0.,0.37638192047117347)-- (-1.,-1.);
\draw [line width=0.8pt] (0.,0.37638192047117347)-- (1.,-1.);
\draw [line width=0.8pt] (0.,0.37638192047117347)-- (1.618033988749895,0.9021130325903065);
\draw (1.5,1.3) node[anchor=north west] {$v_1$};
\draw (-0.5,2.9) node[anchor=north west] {$v_2$};
\draw (-2.6,1.2679725201426373) node[anchor=north west] {$v_3$};
\draw (-1.6,-0.9000098641126053) node[anchor=north west] {$v_4$};
\draw (0.8,-0.8545277161911666) node[anchor=north west] {$v_5$};
\draw [line width=0.8pt] (-1.,-1.)-- (1.,-1.);
\draw [line width=0.8pt] (-1.6180339887498947,0.9021130325903073)-- (-1.,-1.);
\draw [line width=0.8pt] (-1.6180339887498947,0.9021130325903073)-- (0.,2.077683537175253);
\draw [line width=0.8pt,dotted,color=aqaqaq] (0.,2.077683537175253)-- (1.618033988749895,0.9021130325903065);
\draw [line width=0.8pt,dotted,color=aqaqaq] (1.618033988749895,0.9021130325903065)-- (1.,-1.);
\draw (0.08889413098592107,0.6737365518696051) node[anchor=north west] {$v$};
\draw [line width=0.8pt,dotted,color=aqaqaq] (0.,0.37638192047117347) circle (0.18cm);
\begin{scriptsize}
\draw [fill=sqsqsq] (-1.,-1.) circle (1.0pt);
\draw [fill=sqsqsq] (1.,-1.) circle (1.0pt);
\draw [fill=sqsqsq] (1.618033988749895,0.9021130325903065) circle (1.0pt);
\draw [fill=sqsqsq] (0.,2.077683537175253) circle (1.0pt);
\draw [fill=sqsqsq] (-1.6180339887498947,0.9021130325903073) circle (1.0pt);
\draw [fill=sqsqsq] (0.,0.37638192047117347) circle (1.0pt);
\end{scriptsize}
\end{tikzpicture}}
    \caption{The graphs of Lemma \ref{l6}(i)(ii)}
    \label{fig:4}
\end{figure}
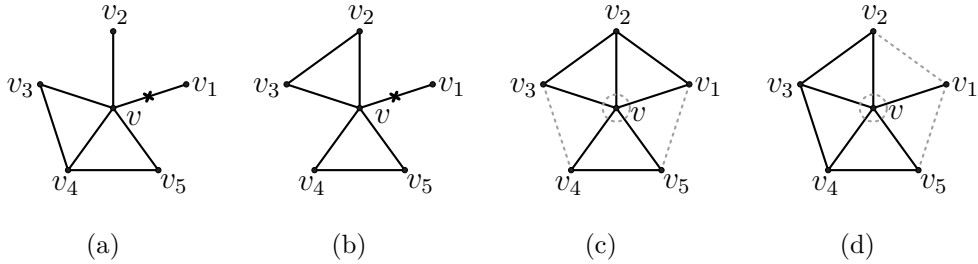
\begin{proof}
(i) Assume by contradiction that $f_3(v)=2$ but $n_3(v)=2$. First, let $[vv_3v_4]$ and $[vv_4v_5]$ be the two $3$-faces incident with $v$ (see Figure \ref{fig:4}(a)). By Lemma \ref{l2}, $d_G(v_i)\ge4$, for each $i$ with $3\le i\le5$. Thus $d_G(v_1)=d_G(v_2)=3$ and $|N^2(v)|\le \sum_{i=1}^{i=5}d_{G}(v_i)-2f_{3}(v)\le 17, |N^2(v_1)|\le \sum_{w\in N_G(v_1)}d_{G}(w)\le 15$, contrary to Lemma \ref{l1}. 
Then $n_3(v)\le1$. Otherwise, we have $v_2v_3,v_4v_5\in E(G)$ (see Figure \ref{fig:4}(b)). By Lemma \ref{l2}, we obtain $d_G(v_i)\ge4$ for each $i$ with $2\le i\le5$, which implies $n_3(v)\le1$.

(ii)  If $f_3(v)=3$, we have either $v_1v_2,v_2v_3,v_4v_5\in E(G)$ or $v_2v_3,v_3v_4,v_4v_5\in E(G)$ (see Figure \ref{fig:4}(c) or (d), respectively). By contradiction, suppose $n_3(v)\ge 1.$ By Lemma \ref{l2}, it should be the case as depicted in Figure \ref{fig:4}(d) and $d_G(v_i)\ge4$ for each $i$ with $2\le i\le5$ in this case. If $d_G(v_1)=3$, then $|N^2(v)|\le \sum_{i=1}^{i=5}d_{G}(v_i)-2f_{3}(v)\le 17$ and $|N^2(v_1)|\le \sum_{w\in N_{G}(v_1)}d_{G}(w)\le 15$, contrary to Lemma \ref{l1}. Thus $n_3(v)=0$. 

(a) Now we are to prove that if $f_3(v)=3, f_4(v)=2$, then $n_4(v)=0$. Let $v_1v_2,v_2v_3,v_4v_5\in E(G)$ (see see Figure \ref{fig:4}(c)). If there is some $i\in\{1,2,3,4,5\}$ with $d_G(v_i)=4$, then let $M=G-v+\{v_1v_5,v_3v_4\}\in \mathcal{G}(v)$. By (\ref{n2}), we have $ |N^2(v)|\le \sum_{i=1}^{i=5}d_{G}(v_i)-2f_{3}(v)-f_{4}(v)\le 16$, a contradiction to Lemma \ref{l0}. Thus $n_4(v)=0$.

Otherwise, we have $v_2v_3,v_3v_4,v_4v_5\in E(G)$ (see Figure \ref{fig:4}(d)) and it is sufficient to verify $d_G(v_i)\neq4$ for each $i$ with $ 1\le i\le3$ by symmetry. First assume that $d_G(v_1)=4.$ Let $M=G-v+\{v_1v_2,v_1v_5\}$. Then $M\in \mathcal{G}(v).$ By (\ref{n2}), we have $ |N^2(v)|= \sum_{i=1}^{i=5}d_{G}(v_i)-2f_{3}(v)-f_{4}(v)\le 16$, a contradiction to Lemma \ref{l0}. Thus $d_{G}(v_1)=5.$ Assume next that $d_G(v_2)=4$. By (1), we have
$|N^2(v)|\le 16$ and $|N^2(v_2)|\le 17$, a contradiction to Lemma \ref{l1}. Thus $d_{G}(v_1)=d_{G}(v_2)=5.$ Assume that $d_{G}(v_3)=4$. By (\ref{n2}), we have $|N^2(v)|\le 16$ and $|N^2(v_3)|\le 15$, a contradiction to Lemma \ref{l1}. Thus $n_4(v)=0$.

(b) Now we are to prove that if $f_3(v)=3$ and $f_4(v)=1$, then $n_4(v)\le1$. Assume by contradiction that $n_4(v)=2$. First assume that  $v_1v_2,v_2v_3,v_4v_5\in E(G)$ (see Figure \ref{fig:4}(c)). Then we can verify $M=G-v+\{v_1v_5,v_3v_4\}\in\mathcal{G}(v)$. By (\ref{n2}), we have $ |N^2(v)|\le\sum_{i=1}^{i=5}d_{G}(v_i)-2f_3(v)-f_4(v)\le 16$, contrary to Lemma \ref{l0}. Thus $n_4(v)\le1$.

Therefore, we may assume that $v_2v_3,v_3v_4,v_4v_5\in E(G)$ (see Figure \ref{fig:4}(d)), and there are $i,j\in\{1,2,3,4,5\}$ such that $d_G(v_i)=d_G(v_j)=4$ in this case.  If $d_G(v_1)=4$, then let $M=G-v+\{v_1v_2,v_1v_5\}\in \mathcal{G}(v)$. By (\ref{n2}), we have $ |N^2(v)|\le \sum_{i=1}^{i=5}d_{G}(v_i)-2f_{3}(v)-f_{4}(v)\le 16$,
a contradiction to Lemma \ref{l0}. Hence $d_G(v_1)=5$. Since $|N^2(v)|\le16$, $|N^2(v_2)|\le 5d_{G}(v_2)-3$ and $|N^2(v_j)|\le 5d_{G}(v_j)-5$ for $j\in \{3,4\}$ (see Figure \ref{fig:5}(d)), by Lemma \ref{l1}, we have $d_G(v_i)\neq4$ for all $i$ with $2\le i\le 4$, a contradiction to $n_4(v)=2$. Thus $n_4(v)\le1$.

\begin{figure}[h]
    \centering
    \subfloat[]{
    \definecolor{aqaqaq}{rgb}{0.6274509803921569,0.6274509803921569,0.6274509803921569}
\definecolor{sqsqsq}{rgb}{0.12549019607843137,0.12549019607843137,0.12549019607843137}
\begin{tikzpicture}[line cap=round,line join=round,>=triangle 45,x=0.6cm,y=0.6cm]
\clip(-2.5,-2.) rectangle (2.5,3.);
\draw [line width=0.8pt] (0.,2.077683537175253)-- (0.,0.37638192047117347);
\draw [line width=0.8pt] (-1.6180339887498947,0.9021130325903073)-- (0.,0.37638192047117347);
\draw [line width=0.8pt] (0.,0.37638192047117347)-- (-1.,-1.);
\draw [line width=0.8pt] (0.,0.37638192047117347)-- (1.,-1.);
\draw [line width=0.8pt] (0.,0.37638192047117347)-- (1.618033988749895,0.9021130325903065);
\draw (1.5,1.3) node[anchor=north west] {$v_1$};
\draw (-0.5,2.9) node[anchor=north west] {$v_2$};
\draw (-2.6,1.2679725201426373) node[anchor=north west] {$v_3$};
\draw (-1.6,-0.9000098641126053) node[anchor=north west] {$v_4$};
\draw (0.8,-0.8545277161911666) node[anchor=north west] {$v_5$};
\draw [line width=0.8pt] (-1.,-1.)-- (1.,-1.);
\draw [line width=0.8pt] (-1.6180339887498947,0.9021130325903073)-- (-1.,-1.);
\draw [line width=0.8pt] (-1.6180339887498947,0.9021130325903073)-- (0.,2.077683537175253);
\draw [line width=0.8pt] (0.,2.077683537175253)-- (1.618033988749895,0.9021130325903065);
\draw [line width=0.8pt,dotted,color=aqaqaq] (1.618033988749895,0.9021130325903065)-- (1.,-1.);
\draw (0.08646136107768357,0.6767045971639347) node[anchor=north west] {$v$};
\draw [line width=0.8pt,dotted,color=aqaqaq] (0.,0.37638192047117347) circle (0.18cm);
\begin{scriptsize}
\draw [fill=sqsqsq] (-1.,-1.) circle (1.0pt);
\draw [fill=sqsqsq] (1.,-1.) circle (1.0pt);
\draw [fill=sqsqsq] (1.618033988749895,0.9021130325903065) circle (1.0pt);
\draw [fill=sqsqsq] (0.,2.077683537175253) circle (1.0pt);
\draw [fill=sqsqsq] (-1.6180339887498947,0.9021130325903073) circle (1.0pt);
\draw [fill=sqsqsq] (0.,0.37638192047117347) circle (1.0pt);
\end{scriptsize}
\end{tikzpicture}}
    \subfloat[]{
    \definecolor{aqaqaq}{rgb}{0.6274509803921569,0.6274509803921569,0.6274509803921569}
\definecolor{sqsqsq}{rgb}{0.12549019607843137,0.12549019607843137,0.12549019607843137}
\begin{tikzpicture}[line cap=round,line join=round,>=triangle 45,x=0.6cm,y=0.6cm]
\clip(-2.5,-2.) rectangle (4.5,4.);
\draw [line width=0.8pt] (0.,2.077683537175253)-- (0.,0.37638192047117347);
\draw [line width=0.8pt] (-1.6180339887498947,0.9021130325903073)-- (0.,0.37638192047117347);
\draw [line width=0.8pt] (0.,0.37638192047117347)-- (-1.,-1.);
\draw [line width=0.8pt] (0.,0.37638192047117347)-- (1.,-1.);
\draw [line width=0.8pt] (0.,0.37638192047117347)-- (1.618033988749895,0.9021130325903065);
\draw (1.7,1.4) node[anchor=north west] {$v_1$};
\draw (-0.5,2.9) node[anchor=north west] {$v_2$};
\draw (-2.6,1.2679725201426373) node[anchor=north west] {$v_3$};
\draw (-1.6,-0.9000098641126053) node[anchor=north west] {$v_4$};
\draw (0.8,-0.8545277161911666) node[anchor=north west] {$v_5$};
\draw (1.1,0.31986394349931585) node[anchor=north west] {$5^+$-face};
\draw [line width=0.8pt] (-1.,-1.)-- (1.,-1.);
\draw [line width=0.8pt] (-1.6180339887498947,0.9021130325903073)-- (-1.,-1.);
\draw [line width=0.8pt] (-1.6180339887498947,0.9021130325903073)-- (0.,2.077683537175253);
\draw [line width=0.8pt] (0.,2.077683537175253)-- (1.618033988749895,0.9021130325903065);
\draw [line width=0.8pt,dotted] (1.618033988749895,0.9021130325903065)-- (1.,-1.);
\draw (0.07810056546902483,0.6313009392753214) node[anchor=north west] {$v$};
\draw [line width=0.8pt,dotted,color=aqaqaq] (0.,0.37638192047117347) circle (0.18cm);
\draw [line width=0.8pt] (1.6180339887498953,2.9021130325903064)-- (1.618033988749895,0.9021130325903065);
\draw [line width=0.8pt] (0.,2.077683537175253)-- (1.6180339887498953,2.9021130325903064);
\draw [line width=0.8pt] (1.6180339887498953,2.9021130325903064)-- (3.23606797749979,2.0776835371752522);
\draw [line width=0.8pt] (3.23606797749979,2.0776835371752522)-- (1.618033988749895,0.9021130325903065);
\draw [line width=0.8pt] (1.618033988749895,0.9021130325903065)-- (3.5201470213402017,0.2840790438404114);
\draw [line width=0.8pt] (3.23606797749979,2.0776835371752522)-- (3.5201470213402017,0.2840790438404114);
\begin{scriptsize}
\draw [fill=sqsqsq] (-1.,-1.) circle (1.0pt);
\draw [fill=sqsqsq] (1.,-1.) circle (1.0pt);
\draw [fill=sqsqsq] (1.618033988749895,0.9021130325903065) circle (1.0pt);
\draw [fill=sqsqsq] (0.,2.077683537175253) circle (1.0pt);
\draw [fill=sqsqsq] (-1.6180339887498947,0.9021130325903073) circle (1.0pt);
\draw [fill=sqsqsq] (0.,0.37638192047117347) circle (1.0pt);
\draw [fill=sqsqsq] (1.6180339887498953,2.9021130325903064) circle (1.0pt);
\draw [fill=sqsqsq] (3.23606797749979,2.0776835371752522) circle (1.0pt);
\draw [fill=sqsqsq] (3.5201470213402017,0.2840790438404114) circle (1.0pt);
\end{scriptsize}
\end{tikzpicture}}
    \subfloat[]{
    \definecolor{sqsqsq}{rgb}{0.12549019607843137,0.12549019607843137,0.12549019607843137}
\begin{tikzpicture}[line cap=round,line join=round,>=triangle 45,x=0.6cm,y=0.6cm]
\clip(-2.7,-2.) rectangle (3,3.);
\draw [line width=0.8pt] (0.,2.077683537175253)-- (0.,0.37638192047117347);
\draw [line width=0.8pt] (-1.6180339887498947,0.9021130325903073)-- (0.,0.37638192047117347);
\draw [line width=0.8pt] (0.,0.37638192047117347)-- (-1.,-1.);
\draw [line width=0.8pt] (0.,0.37638192047117347)-- (1.,-1.);
\draw [line width=0.8pt] (0.,0.37638192047117347)-- (1.618033988749895,0.9021130325903065);
\draw (1.5,1.3) node[anchor=north west] {$v_1$};
\draw (-0.5,2.9) node[anchor=north west] {$v_2$};
\draw (-2.6,1.2679725201426373) node[anchor=north west] {$v_3$};
\draw (-1.6,-0.9000098641126053) node[anchor=north west] {$v_4$};
\draw (0.8,-0.8545277161911666) node[anchor=north west] {$v_5$};
\draw [line width=0.8pt] (-1.,-1.)-- (1.,-1.);
\draw [line width=0.8pt] (-1.6180339887498947,0.9021130325903073)-- (0.,2.077683537175253);
\draw [line width=0.8pt] (0.,2.077683537175253)-- (1.618033988749895,0.9021130325903065);
\draw (0.07267889201058018,0.6331018768425571) node[anchor=north west] {$v$};
\draw (-2.9,0.31986394349931585) node[anchor=north west] {$4$-face};
\draw (0.4,0.31986394349931585) node[anchor=north west] {$5^+$-face};
\begin{scriptsize}
\draw [fill=sqsqsq] (-1.,-1.) circle (1.0pt);
\draw [fill=sqsqsq] (1.,-1.) circle (1.0pt);
\draw [fill=sqsqsq] (1.618033988749895,0.9021130325903065) circle (1.0pt);
\draw [fill=sqsqsq] (0.,2.077683537175253) circle (1.0pt);
\draw [fill=sqsqsq] (-1.6180339887498947,0.9021130325903073) circle (1.0pt);
\draw [fill=sqsqsq] (0.,0.37638192047117347) circle (1.0pt);
\end{scriptsize}
\end{tikzpicture}}
    \subfloat[]{
    \definecolor{aqaqaq}{rgb}{0.6274509803921569,0.6274509803921569,0.6274509803921569}
\definecolor{sqsqsq}{rgb}{0.12549019607843137,0.12549019607843137,0.12549019607843137}
\begin{tikzpicture}[line cap=round,line join=round,>=triangle 45,x=0.6cm,y=0.6cm]
\clip(-2.5,-2.) rectangle (3,3.);
\draw [line width=0.8pt] (0.,2.077683537175253)-- (0.,0.37638192047117347);
\draw [line width=0.8pt] (-1.6180339887498947,0.9021130325903073)-- (0.,0.37638192047117347);
\draw [line width=0.8pt] (0.,0.37638192047117347)-- (-1.,-1.);
\draw [line width=0.8pt] (0.,0.37638192047117347)-- (1.,-1.);
\draw [line width=0.8pt] (0.,0.37638192047117347)-- (1.618033988749895,0.9021130325903065);
\draw (1.5,1.3) node[anchor=north west] {$v_1$};
\draw (-0.5,2.9) node[anchor=north west] {$v_2$};
\draw (-2.6,1.2679725201426373) node[anchor=north west] {$v_3$};
\draw (-1.6,-0.9000098641126053) node[anchor=north west] {$v_4$};
\draw (0.8,-0.8545277161911666) node[anchor=north west] {$v_5$};
\draw (0.4,2.2) node[anchor=north west] {$4$-face};
\draw (0.4,0.31986394349931585) node[anchor=north west] {$5^+$-face};
\draw [line width=0.8pt] (-1.,-1.)-- (1.,-1.);
\draw [line width=0.8pt] (-1.6180339887498947,0.9021130325903073)-- (-1.,-1.);
\draw [line width=0.8pt] (-1.6180339887498947,0.9021130325903073)-- (0.,2.077683537175253);
\draw (0.08889413098592107,0.6737365518696051) node[anchor=north west] {$v$};
\begin{scriptsize}
\draw [fill=sqsqsq] (-1.,-1.) circle (1.0pt);
\draw [fill=sqsqsq] (1.,-1.) circle (1.0pt);
\draw [fill=sqsqsq] (1.618033988749895,0.9021130325903065) circle (1.0pt);
\draw [fill=sqsqsq] (0.,2.077683537175253) circle (1.0pt);
\draw [fill=sqsqsq] (-1.6180339887498947,0.9021130325903073) circle (1.0pt);
\draw [fill=sqsqsq] (0.,0.37638192047117347) circle (1.0pt);
\end{scriptsize}
\end{tikzpicture}}
    \caption{The graphs of Lemma \ref{l6}(iii)}
    \label{fig:5}
\end{figure}
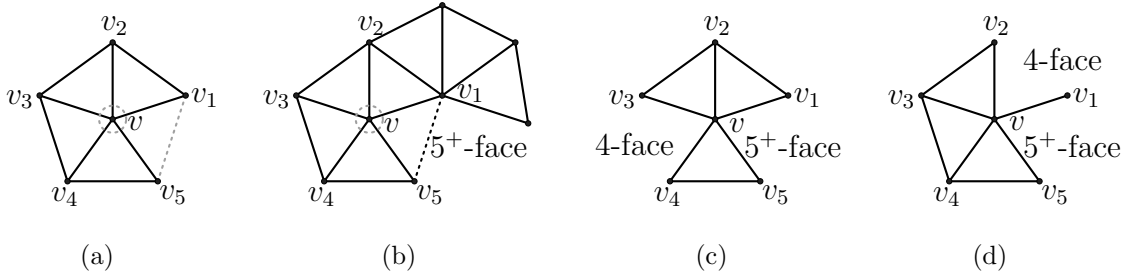

(iii)  First we prove that $f_4(v)=0$ if $f_3(v)=4$. Assume that $v$ is contained in four triangles and one $4$-face. Let $M=G-v+\{v_1v_5\}$ (see Figure \ref{fig:4}(a)). By simple computation, we have $M\in \mathcal{G}(v)$. By (\ref{n2}), since $ |N^2(v)|\le \sum_{i=1}^{i=5}d_{G}(v_i)-2f_{3}(v)-f_{4}(v)\le 16$, there is a contradiction to Lemma \ref{l0}. Thus $f_4(v)=0$. Now we prove $d_G(v_i)=5$ for each $i$ with $1\le i\le 5$. Suppose there is some $i\in\{1,2,3,4,5\}$ such that $d_G(v_i)\le4$. Take $M=G-v+\{v_1v_5\}\in \mathcal{G}(v)$ (see Figure \ref{fig:5}(a)). By (\ref{n2}), $ |N^2(v)|\le \sum_{i=1}^{i=5}d_{G}(v_i)-2f_{3}(v)\le 16$, contrary to Lemma \ref{l0}. Thus $d_G(v_i)=5$ for all $i$ with $1\le i\le 5$.

Furthermore, $v_1$ and $v_5$ are corners of $v$ by definition. Assume by contradiction that $v_1$ is a bad $5$-vertex (see Figure \ref{fig:5}(b)). Let $M=G-v+\{v_1v_5\}$. Then $M\in \mathcal{G}(v)$. Since $ |N^2(v)|\le \sum_{i=1}^{i=5}d_{G}(v_i)-2f_{3}(v)-1\le 16$, we obtain a contradiction to Lemma \ref{l0}.\end{proof}  

\subsection{Discharging}
Applying the properties stated above in Section 3.1, we can successfully prove our argument by applying a discharging method on the graph $G$. First assign initial charges to each vertex $x\in V(G)$ and face $f\in F(G)$ by setting $\mu(x)=d_{G}(x)-4$ and $\mu(f)=d_{G}(f)-4$ respectively. By Euler formula, it follows that
\begin{equation}\label{euler}
\sum_{x\in V(G)\cup F(G)}\mu(x)=\sum_{v\in V(G)}(d_{G}(v)-4)+\sum_{f\in F(G)}(d_{G}(f)-4)=-8.
\end{equation}

Then we define discharging rules and redistribute the charges while maintaining their total sum, so that all new charges are non-negative. Finally, we use $\mu'(x)$ to represent the final charge of each $x\in V(G)\cup F(G)$ after the completion of the discharge procedure. The discharging rules we will apply is defined as below.

{\bf{Discharging Rules:}} 

{\em\indent ($R_1$) Every $3$-vertex $v$ receives charge $\frac{1}{5}$ from each $5$-neighbor.}

{\em\indent ($R_2$) Every $4$-vertex $v$ with $f_3(v)=1$ receives charge $\frac{1}{15}$ from each $5$-neighbor.}

{\em\indent ($R_3$) Every bad $5$-vertex $v$ receives charge $\frac{1}{15}$ from each corner of $v$.}

{\em\indent ($R_4$) Every $3$-face receives charge $\frac{1}{3}$ from each incident $4^+$-vertex.}

{\em\indent ($R_5$) Every $5^+$-face gives charge $\frac{1}{5}$ to each incident vertex.}
\begin{figure}[h]
    \centering
    \subfloat[]{
    \begin{tikzpicture}[line cap=round,line join=round,>=triangle 45,x=0.7cm,y=0.7cm]
\clip(-1.3,-1) rectangle (1.5,3);
\draw [line width=0.8pt] (-1.,0.)-- (1.,0.);
\draw [shift={(-0.0036522668476274118,-0.8761100795950977)},line width=0.8pt]  plot[domain=0.7176518300992938:2.2330465153414476,variable=\t]({1.*1.3322487546687778*cos(\t r)+0.*1.3322487546687778*sin(\t r)},{0.*1.3322487546687778*cos(\t r)+1.*1.3322487546687778*sin(\t r)});
\draw [line width=2.pt] (-0.7067052519898682,0.40679016339009066)-- (-0.8228426067390476,0.1745154156909671);
\draw [line width=2.pt] (-0.8228426067390476,0.1745154156909671)-- (-0.5583075209214722,0.1745154156909671);
\draw (0.9,0.1) node[anchor=north west] {$5$};
\draw (-0.4,1.5) node[anchor=north west] {$\frac{1}{5}$};
\draw (-1.3,0.1) node[anchor=north west] {$3$};
\begin{scriptsize}
\draw [fill=black] (-1.,0.) circle (1.0pt);
\draw [fill=black] (1.,0.) circle (1.0pt);
\end{scriptsize}
\end{tikzpicture}}
    \subfloat[]{
    \begin{tikzpicture}[line cap=round,line join=round,>=triangle 45,x=0.5cm,y=0.5cm]
\clip(-2.5,-3.) rectangle (3.,3.);
\draw [line width=0.8pt] (0.,2.)-- (0.,-2.);
\draw [line width=0.8pt] (-2.,0.)-- (2.,0.);
\draw (0.05954862249607203,1) node[anchor=north west] {$4$};
\draw (-0.4,3) node[anchor=north west] {$5$};
\draw (-0.4,-1.9) node[anchor=north west] {$5$};
\draw [shift={(0.2389324885448464,1.1292506137611262)},line width=0.8pt]  plot[domain=1.8386034161653602:3.9520267206560247,variable=\t]({1.*0.902935893470632*cos(\t r)+0.*0.902935893470632*sin(\t r)},{0.*0.902935893470632*cos(\t r)+1.*0.902935893470632*sin(\t r)});
\draw [line width=1.2pt] (-0.8904929632235269,0.5383876125887598)-- (-0.38335646311593863,0.4749955500753113);
\draw [line width=1.2pt] (-0.3072859880998004,0.906061575166761)-- (-0.38335646311593863,0.4749955500753113);
\draw [line width=0.8pt] (0.,2.)-- (2.,0.);
\draw [shift={(-0.2389324885448464,-1.1292506137611262)},line width=0.8pt]  plot[domain=-1.3029892374244332:0.8104340670662319,variable=\t]({1.*0.902935893470632*cos(\t r)+0.*0.902935893470632*sin(\t r)},{0.*0.902935893470632*cos(\t r)+1.*0.902935893470632*sin(\t r)});
\draw [line width=1.2pt] (0.8904929632235269,-0.5383876125887598)-- (0.38335646311593863,-0.4749955500753113);
\draw [line width=1.2pt] (0.3072859880998004,-0.906061575166761)-- (0.38335646311593863,-0.4749955500753113);
\draw (-1.8,2.3) node[anchor=north west] {$\frac{1}{15}$};
\draw (0.5,-0.6) node[anchor=north west] {$\frac{1}{15}$};
\begin{scriptsize}
\draw [fill=black] (0.,0.) circle (1.0pt);
\draw [fill=black] (-2.,0.) circle (1.0pt);
\draw [fill=black] (2.,0.) circle (1.0pt);
\draw [fill=black] (0.,2.) circle (1.0pt);
\draw [fill=black] (0.,-2.) circle (1.0pt);
\end{scriptsize}
\end{tikzpicture}}
    \subfloat[]{
    \definecolor{sqsqsq}{rgb}{0.12549019607843137,0.12549019607843137,0.12549019607843137}
\begin{tikzpicture}[line cap=round,line join=round,>=triangle 45,x=0.7cm,y=0.7cm]
\clip(-2.5,-2.) rectangle (3,3.);
\draw [line width=0.8pt] (0.,2.077683537175253)-- (0.,0.37638192047117347);
\draw [line width=0.8pt] (-1.6180339887498947,0.9021130325903073)-- (0.,0.37638192047117347);
\draw [line width=0.8pt] (0.,0.37638192047117347)-- (-1.,-1.);
\draw [line width=0.8pt] (0.,0.37638192047117347)-- (1.,-1.);
\draw [line width=0.8pt] (0.,0.37638192047117347)-- (1.618033988749895,0.9021130325903065);
\draw [line width=0.8pt] (-1.,-1.)-- (1.,-1.);
\draw [line width=0.8pt] (-1.6180339887498947,0.9021130325903073)-- (-1.,-1.);
\draw [line width=0.8pt] (-1.6180339887498947,0.9021130325903073)-- (0.,2.077683537175253);
\draw [line width=0.8pt] (0.,2.077683537175253)-- (1.618033988749895,0.9021130325903065);
\draw (1.5,1.3) node[anchor=north west] {$5$};
\draw (-0.5,2.9) node[anchor=north west] {$5$};
\draw (-2.4,1.2679725201426373) node[anchor=north west] {$5$};
\draw (-1.6,-0.9000098641126053) node[anchor=north west] {$5$};
\draw (0.9,-0.8545277161911666) node[anchor=north west] {$5$};
\draw (-0.1,1.3) node[anchor=north west] {$5$};
\draw (0.7,0.0) node[anchor=north west] {$5^+$-face};
\draw [shift={(0.796268360033486,0.9963044794237153)},line width=0.8pt]  plot[domain=4.212093129369753:6.167375596146971,variable=\t]({1.*0.834434700280234*cos(\t r)+0.*0.834434700280234*sin(\t r)},{0.*0.834434700280234*cos(\t r)+1.*0.834434700280234*sin(\t r)});
\draw [line width=1.2pt] (0.39600242457082657,0.2641375244605961)-- (0.7138760144328892,0.359499601419215);
\draw [line width=1.2pt] (0.39600242457082657,0.2641375244605961)-- (0.5973223648167996,-0.03254449274399617);
\draw (1.2,0.8) node[anchor=north west] {$\frac{1}{15}$};
\begin{scriptsize}
\draw [fill=sqsqsq] (-1.,-1.) circle (1.0pt);
\draw [fill=sqsqsq] (1.,-1.) circle (1.0pt);
\draw [fill=sqsqsq] (1.618033988749895,0.9021130325903065) circle (1.0pt);
\draw [fill=sqsqsq] (0.,2.077683537175253) circle (1.0pt);
\draw [fill=sqsqsq] (-1.6180339887498947,0.9021130325903073) circle (1.0pt);
\draw [fill=sqsqsq] (0.,0.37638192047117347) circle (1.0pt);
\end{scriptsize}
\end{tikzpicture}}
    \subfloat[]{
    \begin{tikzpicture}[line cap=round,line join=round,>=triangle 45,x=0.7cm,y=0.7cm]
\clip(-2.,-1.) rectangle (2.,2.5);
\draw [line width=0.8pt] (0.,1.7320508075688776)-- (-1.,0.);
\draw [line width=0.8pt] (1.,0.)-- (-1.,0.);
\draw [line width=0.8pt] (1.,0.)-- (0.,1.7320508075688776);
\draw [->,line width=0.8pt] (0.,1.7320508075688776) -- (0.,0.86);
\draw (-0.4,1.0) node[anchor=north west] {$\frac{1}{3}$};
\draw (-0.6,2.52) node[anchor=north west] {$4,5$};
\begin{scriptsize}
\draw [fill=black] (-1.,0.) circle (1.0pt);
\draw [fill=black] (1.,0.) circle (1.0pt);
\draw [fill=black] (0.,1.7320508075688776) circle (1.0pt);
\end{scriptsize}
\end{tikzpicture}}
    \subfloat[]{
    \begin{tikzpicture}[line cap=round,line join=round,>=triangle 45,x=0.6cm,y=0.6cm]
\clip(-2.,-2.) rectangle (3.,2.5);
\draw [line width=0.8pt] (0.,1.)-- (2.,1.);
\draw [line width=0.8pt] (0.,1.)-- (-1.,-0.7320508075688774);
\draw (-0.2,-0.18911394639296739) node[anchor=north west] {$5^+$-face};
\draw (-1.2,2) node[anchor=north west] {$3,4,5$};
\draw [->,line width=0.8pt] (0.4400706369874669,-0.1342633362877752) -- (0.,1.);
\draw (0.2676027701585017,1.0945703378990466) node[anchor=north west] {$\frac{1}{5}$};
\begin{scriptsize}
\draw [fill=black] (0.,1.) circle (1.0pt);
\draw [fill=black] (2.,1.) circle (1.0pt);
\draw [fill=black] (-1.,-0.7320508075688774) circle (1.0pt);
\end{scriptsize}
\end{tikzpicture}}
    
    \caption{$R_1$-$R_5$}
    \label{fig:my_label}
\end{figure}

Now we will derive the final contradiction by proving each face and vertex has a non-negative new charge after discharging.
\begin{claim}\label{cl7}
    Each face in $G$ has a non-negative charge after discharging.
\end{claim}
\begin{pfc}
We check that any face of $G$ has a non-negative charge case by case as below. 

Case 1. $d_G(f)=3$.

The initial charge of the $3$-face $f\in F(G)$ is $\mu(f)=d_G(f)-4=-1$. By Lemma \ref{l2}, $3$-vertex is not contained in $3$-face $f$. By $(R_4)$, $3$-face $f$ obtains charges $\frac{1}{3}\times3$ from each $4^+$-vertex contained in $3$-face. Thus $\mu'(f)=-1+\frac{1}{3}\times3=0$.

Case 2. $d_G(f)=4$.

 Note that $4$-face $f$ is not involved in any discharging process in $(R_1)-(R_5)$. Then  $\mu'(f)=\mu(f)=d_G(f)-4=0$.
 
Case 3. $d_G(f)\ge 5$.

The original charge of $5^{+}$-face $f$ is $\mu(f)=d_G(f)-4\ge 1$. By $(R_5)$, $5^{+}$-face $f$ gives charges $\frac{1}{5}$ to every vertex contained in $f$. Then $\mu'(f)\ge d_G(f)-4-\frac{1}{5}\times d_G(f)\ge 0$. 

Thus we gain $\mu'(f)\ge 0$ for every $ f\in F(G)$, which completes the argument.
\end{pfc}
\begin{claim}\label{cl7}
    Each vertex in $G$ has a non-negative charge after discharging.
\end{claim}
\begin{pfc}
Now we are to check that any vertex $v\in V(G)$ has a non-negative charge after discharging case by case.

Case 1. $d_G(v)=3$.

  For each $3$-vertex $v\in V(G)$, we have $\mu(v)=d_G(v)-4=-1$. By Lemma \ref{l2}, $3$-vertex $v$ is not contained in any $3$-face. By Lemma \ref{l3}, the neighbors of $v$ are all $5$-vertices and $v$ is contained in at most one $4$-face. Then $n_5(v)=3$ and $f_{5^+}(v)\ge2$. By $(R_1)$ and $(R_5)$, $\mu'(v)=\mu(v)+\frac{1}{5}\times(f_{5^+}(v)+n_5(v))\ge 0$.
  
Case 2. $d_G(v)=4$.

  For every $4$-vertex $v\in V(G)$,  we have $\mu(v)=d_G(v)-4=0$. By Lemma \ref{l3}, all the neighbors of $3$-vertex are $5$-vertex, which means $n_3(v)=0$. By Lemma \ref{l4}$(i)$, there is $f_3(v)\le1$. Then we have the following two cases to discuss.

Subcase 2.1 $f_3(v)=0$. 

Obviously we have $f_{4^+}(v)=4$ and so $\mu'(v)\ge\mu(v)\ge0$.

Subcase 2.2 $f_3(v)=1$.

In this case, $v$ is incident with one $3$-face, by Lemma \ref{l2}, and then we have $n_5(v)\ge1$. By Lemma \ref{l5}(i), we have $f_4(v)\le2$. When $f_4(v)\le1$, there is 
$\mu'(v)\ge\mu(v)+\frac{1}{15}-\frac{1}{3}+\frac{2}{5}\ge\frac{2}{15}$ by $(R_2), (R_4),(R_5)$. Otherwise $f_4(v)=2$, we have $n_5(v)=4$ by Lemma \ref{l5} (ii). By $(R_2),(R_4)$ and $(R_5)$, we have $\mu'(v)\ge\mu(v)+\frac{4}{15}-\frac{1}{3}+\frac{1}{5}\ge\frac{2}{15}$.

Thus for any $4$-vertex $v\in V(G)$, $\mu'(v)\ge0$ always holds.

Case 3. $d_G(v)=5$

For each $5$-vertex $v$ contained in $G$, the original charge is $\mu(v)=d_G(v)-4=1$. Obviously we have $n_3(v)+n_4(v)+t_5(v)\le5$. If $v$ is a $5$-vertex with $f_3(v)\le 3$, by $(R_1)-(R_5)$, then we have
\begin{equation}\label{n3}
    \mu'(v)\ge  \mu(v)+\frac{1}{5}(f_{5^+}(v)-n_3(v))-\frac{1}{3}f_3(v)-\frac{1}{15}(n_4(v)+t_5(v)).
\end{equation} 
By Lemma \ref{l4}, $v$ is contained in at most $4$ triangles. Then we have the following cases to verify.

Subcase 3.1 $f_3(v)=0$

 Then $\mu'(v)\ge1-\frac{1}{5}\times(n_3(v)+n_4(v)+t_5(v))\ge0$ by (\ref{n3}).

Subcase 3.2 $f_3(v)=1$

In this case, $v$ is contained in one triangle. By Lemma \ref{l2}, $3$-vertex is not incident with any $3$-face, that is $n_3(v)\le3$. If $f_4(v)\le3$, then $f_{5^+}(v)\ge1$. By (\ref{n3}), we have $\mu'(v)\ge1+\frac{1}{5}\times(1-3)-\frac{1}{3}-\frac{1}{15}\times2\ge\frac{2}{15}$.  By Lemma \ref{l3}, $3$-vertex is contained in at most one $4$-face. Hence, if $f_4(v)=4$, then $n_3(v)=0$. By (\ref{n3}), we gain $\mu'(v)\ge1-\frac{1}{3}-\frac{1}{15}\times5\ge\frac{1}{3}$.

Subcase 3.3 $f_3(v)=2$

By Lemma \ref{l6}(i), we obtain $n_3(v)\le1$. When $f_4(v)\le2$, by (\ref{n3}), there is $\mu'(v)\ge1+\frac{1}{5}\times(1-1)-\frac{1}{3}\times2-\frac{1}{15}\times4\ge\frac{1}{15}$. When $f_4(v)=3$, by Lemma \ref{l3}, we have $n_3(v)=0$. By (\ref{n3}),  $\mu'(v)\ge1-\frac{1}{3}\times2-\frac{1}{15}\times5\ge0$. 

Subcase 3.4 $f_3(v)=3$

By Lemma \ref{l6}(ii), we have $n_3(v)=0$. If $f_4(v)=0$, we have $n_4(v)+t_5(v)\le5$ and then $\mu'(v)\ge1+\frac{1}{5}\times2-\frac{1}{3}\times3-\frac{1}{15}\times5\ge\frac{1}{15}$  by (\ref{n3}). When $f_4(v)=1$, by Lemma \ref{l6}(ii)(b), there is $n_4(v)\le1$. By the definition of $t_5(v)$, we have $t_5(v)\le2$(see Figure \ref{fig:5}(c): $t_5(v)=2$ when $v$ is a corner of each of the bad $5$-vertiex $v_1,v_5$, or Figure \ref{fig:5}(d): $t_5(v)=1$ when $v$ is the only corner of the bad $5$-vertex $v_5$). By (\ref{n3}), we get $\mu'(v)\ge1+\frac{1}{5}-\frac{1}{3}\times3-\frac{1}{15}\times3\ge0$. If $f_4(v)=2$, then we obtain $n_5(v)=5$ by Lemma \ref{l6}(ii)(a). However, by the definition of $t_5(v)$, $t_5(v)=0$ holds in this case. Thus we gain $\mu'(v)\ge1-\frac{1}{3}\times3\ge0$ by (\ref{n3}).

Subcase 3.5 $f_3(v)=4$

In this case, $v$ is a bad $5$-vertex. By Lemma \ref{l6}(iii), we have $n_5(v)=5$ and $f_4(v)=0$, which implies $f_{5^+}(v)=1$. For the vertex $v$, as $v_1,v_5$ are corners of $v$ (see Figure \ref{fig:5}(b)), it follows that $v$ obtains charge $\frac{1}{15}\times 2$ from $v_1,v_5$ by $(R_3)$.  By the definition of corner of bad $5$-vertex, it is straightforward to have that $v$ is not the corner of each $v_i$ with $2\le i\le 4$. By Lemma \ref{l6}(iii), $v_1,v_5$ are not bad $5$-vertices which imply that $v$ is not corner of $v_1,v_5$.  By (R3), any vertex $v_i$ with $1\le i\le 5$ does not receive charge from $v$. By $(R_3), (R_4), (R_5)$,  $\mu'(v)\ge1+\frac{1}{5}-\frac{1}{3}\times4+\frac{1}{15}\times2\ge0$.

Thus $\mu'(v)\ge 0$ for each $ v\in V(G)$.

By Claims $1$ and $2$, each vertex and face has a non-negative final charge. Hence total final charge is non-negative, which is a contradiction to (\ref{euler}). The proof is completed. \end{pfc}

\section{Concluding Remarks}
This paper has proved that every planar graph $G$ with $\Delta(G)\le 5$ admits a square coloring with $\chi(G^2)\le 17.$ Actually, applying the same method, we could get the same upper bound for the list version of Theorem \ref{main}. It would be interesting to research further to obtain the optimal upper bound. After completing the writing of this paper, we learned that Aoki also proved Theorem \ref{main} independently in \cite{Aoki2023}.  For more related open problems on square coloring, We refer interested readers to \cite{Cranston2023} 


\section*{Acknowledgments}
This research is partially supported by National Natural Science Foundation of China
(No. 11901434).

 \end{document}